\documentclass[oneside,english, 11pt]{amsart}
\usepackage[T1]{fontenc}
\usepackage[utf8]{inputenc}
\usepackage{enumerate}
\usepackage{amsthm}
\usepackage{amstext}
\usepackage{amssymb}
\usepackage{amsmath}
\usepackage{mathtools}
\makeatletter
\numberwithin{equation}{section}
\numberwithin{figure}{section}
\theoremstyle{plain}
\newtheorem{thm}{\protect\theoremname}
  \theoremstyle{plain}
  \newtheorem{pro}[thm]{\protect\propositionname}
  \theoremstyle{remark}
  
  \theoremstyle{definition}
  
  \theoremstyle{plain}
  
    \theoremstyle{definition}
  \newtheorem{exa}[thm]{\protect\examplename}

\theoremstyle{remark}

\usepackage[english]{babel}
  \providecommand{\definitionname}{Definition}
  \providecommand{\lemmaname}{Lemma}
  \providecommand{\propositionname}{Proposition}
  \providecommand{\remarkname}{Remark}
\providecommand{\theoremname}{Theorem}
\providecommand{\examplename}{Example}
\providecommand{\assertionname}{Assertion}

\DeclareMathOperator{\cs}{CS}

\DeclareMathOperator{\bb}{BB}
\DeclareMathOperator{\gsv}{GSV}



\author{Percy Fernández}
\author{Liliana Puchuri}
\author{Rudy Rosas}

\begin{document}

\title{Foliations on $\mathbb{P}^2$ with only one singular point}

\maketitle

\begin{abstract} In this paper we study holomorphic foliations on $\mathbb{P}^2$ with only one singular point. If the singularity has algebraic multiplicity one, we prove that the foliation has no
invariant algebraic curve. We also present several examples of such foliations in degree three.
\end{abstract}
\section{introduction}

The problem of classification of holomorphic foliations in the complex projective plane 
has been the object of intense study in the last decades --- see for example \cite{B,OC,CeLi} and the references therein. The special case of foliations of degree $0$ and $1$ is well understood and
the classification  was developed in \cite{CDbook, Jou}. 
 The case of foliations of degree $2$ is much more difficult and we only have partial classifications in some special cases, for example if the foliation has a Morse type center singularity
(\cite{CeLi,D}) or if the foliation has a unique singularity in the projective plane (\cite{CD}).  In this paper, we study foliations of degree $d\geq 3$ with a unique singularity in  $\mathbb{P}^2$ and  our main result is the following:

 \begin{thm}\label{teorema1}Let $\mathcal{F}$ be a holomorphic foliation of degree $d\ge 1$  on $\mathbb{P}^2$ with only one singularity. If this singularity has algebraic multiplicity one, then $\mathcal{F}$ has no invariant algebraic curve. 
\end{thm}
  We point out that these result was already obtained in \cite{ClP} when the singularity is a saddle-node. Anyway, for the sake of completeness we also include a proof of this case. The paper is organized as follows. In Section \ref{pre} we recall some basic definitions and results about singular holomorphic foliations with special attention to the indices associated to singularities and separatrices. In Section \ref{nil} we review the classification of nilpotent singularities, their reduction processes and the computation of some of their indices. In Section \ref{pro} we prove Theorem \ref{teorema1}. Finally, in Section \ref{exa} we present several
explicit examples of foliations  with a unique singularity.

 \section{Preliminaries} \label{pre}
 In this section we recall some general definitions and results about singular holomorphic foliations.  Since we focus on what we really 
 need along the paper, we refer the reader to 
 \cite{blibro} and the references therein for more details and proofs. 
 
  Let $\mathcal{F}$ be  a singular holomorphic foliation with isolated singularities on a complex surface. If $p$ is a singular point of
 $\mathcal{F}$ we can associate two main indices to $p$. The first one, called the Milnor number of $\mathcal{F}$ at $p$ and denoted
 by $\mu( \mathcal{F},p)$, is simply
 the Poincaré--Hopf index of any holomorphic vector field defining $\mathcal{F}$ on a neighborhood of $p$. So this index is a non-negative integer number and has a topological nature.
  The other index,  called the Baum--Bott  index of  $\mathcal{F}$ at $p$ and denoted by $\bb(\mathcal{F},p)$, is in general
   a complex number and has no topological meaning. If $\mathcal{F}$ is  defined on $\mathbb{P}^2$ and has degree $d$, we have the following well known formulas:
   \begin{align}\sum_p \mu( \mathcal{F},p) =d^2+d+1\label{Milnorsum}\end{align}
   \begin{align}\label{BaumBottsum}\sum_p \bb(\mathcal{F},p)=(d+2)^2
\end{align} where the summations are taken over the singularities of $\mathcal{F}$.
There are some indices that we can associate to a germ of  curve at $p$ that is invariant by $\mathcal{F}$. A reduced germ 
$C$ of  analytic curve at $p$ is said to be invariant by $\mathcal{F}$ if each connected component of $C\backslash\{p\}$ is contained in a leaf of $\mathcal{F}$; the invariant curve $C$ is called a separatrix of $\mathcal{F}$ if $C$ is irreducible.  If $C$ is a reduced germ
at $p$ which is invariant by $\mathcal{F}$, we are mainly interested in two indices associated to $(\mathcal{F},C,p)$. The first one, called the GSV index and 
denoted by $\gsv(\mathcal{F},C,p)$, is an integer number, while the second one, called the Camacho--Sad index and denoted by 
$\cs(\mathcal{F},C,p)$, is a complex number. If $\mathcal{F}$ is a foliation of degree $d$ on $\mathbb{P}^2$ and $C\subset\mathbb{P}^2$ is a
 reduced algebraic curve of degree $m$  that is invariant by $\mathcal{F}$, 
we have the following formulas:
\begin{align}\label{GSVformula}\sum_{p} \gsv( \mathcal{F},C,p) =m(d+2)-m^2\end{align}
   \begin{align}\label{CSformula}\sum_p \cs(\mathcal{F},C,p)=m^2\end{align}
where the summations are taken over the singularities of $\mathcal{F}$ that are contained in $C$. If $S$ is a
smooth separatrix  of $\mathcal{F}$ 
at $p$, the index $\gsv(\mathcal{F},S,p)$ is simply the Poincaré-Hopf index at $p$ of the restriction to $S$ of a local
holomorphic vector field $v$ defining $\mathcal{F}$:   if $\psi\colon(\mathbb{C},0 )\to (S,p)$ is a
holomorphic  
parametrization, the index $\gsv(\mathcal{F}, S,p)$ is the vanishing order of $\psi^*(v)$ at $0\in\mathbb{C}$.
If $S$ is not smooth and $\psi$ is a Puiseux parametrization, we can also consider the the vanishing order of $\psi^*(v)$ at $0\in\mathbb{C}$, which will be denoted by $\mu(\mathcal{F},S,p)$ and called the multiplicity of $\mathcal{F}$ at $(S,p)$. We remark that in general $\mu(\mathcal{F},S,p)\neq \gsv(\mathcal{F},S,p)$ if the separatrix $S$ is not smooth. If $C\subset\mathbb{P}^2$ is
a reduced algebraic curve of degree $m$ that is invariant by $\mathcal{F}$,  we have that (see \cite{CL})
\begin{align}\label{CLformula}\sum\limits_{B} \mu(\mathcal{F},B,p)=\chi (\tilde{C}) +m(d-1),
\end{align} where $\chi(\tilde{C})$ is the Euler characteristic of the normalization $\tilde{C}$  of $C$ and the summation is taken over all
local branches of $C$ passing through the singularities of $\mathcal{F}$ in $C$. This formula is essentially the same as
that in
\eqref{GSVformula}, but \eqref{CLformula} has the advantage of putting in evidence the topological type of the curve. 

We recall now some facts about the behavior of the indices above under blow-ups. Let $\pi\colon (M,D)\to (\mathbb{C}^2,0)$ be the blow-up at the origin,
were $D$ stands for the associated exceptional divisor $\pi^{-1}(0)$. Suppose that $\mathcal{F}$ is a holomorphic foliation in $(\mathbb{C}^2,0)$ and let $\tilde{\mathcal{F}}$ denote the strict transform of $\mathcal{F}$ by $\pi$. Then, if $\mathcal{F}$  has algebraic multiplicity $\nu_\mathcal{F}$ at the origin and $D$ is invariant by $\tilde{\mathcal{F}}$, we have 
\begin{align}\label{Milnor-blow-up} \mu(\mathcal{F},0)=\sum_p\mu(\tilde{\mathcal{F}},p) +\nu_\mathcal{F}^2-\nu_\mathcal{F}-1
\end{align}
and \begin{align}\label{BB-blow-up} \bb(\mathcal{F},0)=\sum_p\bb(\tilde{\mathcal{F}},p) +\nu_\mathcal{F}^2,
\end{align} where the summations are taken over the singularities of $\tilde{\mathcal{F}}$ in the exceptional divisor $D$.  
Let $S$ be a separatrix of $\mathcal{F}$ and let $\tilde{S}$ be the strict transform of $S$ by $\pi$; the curve $\tilde{S}$ is a separatrix of $\tilde{\mathcal{F}}$ at some singular point $p$ in $D$. Then, if $\nu_S$ is the algebraic multiplicity of $S$ at the origin,  we have 
\begin{align}\label{CS-blow-up} \cs (\mathcal{F},S,0)=\cs(\tilde{\mathcal{F}},\tilde{S},p)+\nu_S^2
\end{align} and,  if $D$ is invariant by $\tilde{\mathcal{F}}$, 
\begin{align} \gsv (\mathcal{F},S,0)=\gsv(\tilde{\mathcal{F}},\tilde{S},p)+\nu_S\nu_\mathcal{F}-\nu_S^2,
\end{align} which becomes \begin{align}\label{GSV-blow-up} \gsv (\mathcal{F},S,0)=\gsv(\tilde{\mathcal{F}},\tilde{S},p)+\nu_\mathcal{F}-1,\end{align}  if $S$ is smooth. 

\subsection*{Reduced singularities}
Consider a holomorphic vector field $X$ defining $\mathcal{F}$ in $(\mathbb{C}^2,0)$. If the origin is a singularity of $\mathcal{F}$, we say that it is non-degenerate if the derivative $dX(0)$ is non-singular. A non-degenerate singularity has Milnor number one, so these singularities are simple in a topological sense. In  an analytical or even formal viewpoint, the simplest singularities are the so called
reduced:  $\mathcal{F}$ is said to be reduced at $0\in\mathbb{C}^2$ if the eigenvalues $\lambda_{1},\lambda_{2}$ of $dX(0)$ satisfy $\lambda_{1}\neq0$ and $\lambda_{2}/\lambda_{1}\not\in\mathbb{Q}_{>0}$. If $\lambda_2\neq 0$ the singularity is clearly non-degenerate; otherwise we say the singularity is a saddle-node. According to the well known reduction theorem for holomorphic foliations --- see \cite{S} or \cite{blibro} --- given a germ of singular holomorphic foliation in $(\mathbb{C}^2,0)$,  after a suitable finite number blow-ups we can obtain a complex surface with a holomorphic foliation whose singularities are all reduced.

A non-degenerate reduced singularity has exactly two separatrices, which are respectively tangent to each eigenspace of $dX(0)$. At a formal level, the same happens with a saddle-node: these singularities have exactly two formal separatrices, which are tangent to each eigenspace. The separatrix that is tangent to the eigenspace associated to the non-zero eigenvalue is called strong and it is always convergent; the other one is called weak and it can be purely formal. Suppose that $\mathcal{F}$ is reduced non-degenerate at $0\in\mathbb{C}^2$ and let $\lambda_1$ and $\lambda_2$ be the eigenvalues of $DX(0)$. Let $S_1$ and $S_2$ be the separatrices associated to the $\lambda_1$--eigenspace and 
the $\lambda_2$--eigenspace, respectively. Then we have 
\begin{align*}&\cs(\mathcal{F},S_1,0)=\frac{\lambda_2}{\lambda_1},\quad \cs(\mathcal{F},S_2,0)=\frac{\lambda_1}{\lambda_2},\\      
  & \gsv(\mathcal{F},S_1,0)= \cs(\mathcal{F},S_2,0)=1,\\
&\bb(\mathcal{F},0)=\frac{\lambda_2}{\lambda_1}+\frac{\lambda_1}{\lambda_2}+2.\end{align*} We point out that the last equality holds even if $\mathcal{F}$ is only non-degenerate. Suppose now that $\mathcal{F}$ is a saddle-node and let $S_s$ and $S_w$ 
be
 its strong and weak separatrices, respectively. We remark that the definitions of  GSV and CS indices  given in \cite{blibro} or \cite{B} work  in the same way for purely formal separatrices. Then, if we set $\mu(\mathcal{F},0)=\mu$,  we have 
\begin{align*}&\cs(\mathcal{F},S_s,0)=0, \quad  \gsv(\mathcal{F},S_s,0)= 1,\\
& \gsv(\mathcal{F},S_w,0)= \mu\ge 2.\end{align*}
On the other hand, the index $\lambda=\cs(\mathcal{F},S_w,0)$ can take any complex value and we have 
\begin{align}\bb(\mathcal{F},0)=2\mu+\lambda. \label{saddle-node-bb}\end{align}

\section{Nilpotent singularities}\label{nilpotent}\label{nil}
A especial kind of singularities of algebraic multiplicity one that are not reduced are the nilpotent ones: those that are generated by  holomorphic vector fields with nilpotent linear part. In this section we recall  some facts about nilpotent singularities and their reduction processes, and we include  computations of  some of their indices. For more details and proofs we refer the reader to \cite{BMS,CM,Me,MS}. In suitable holomorphic coordinates, a nilpotent singularity $\mathcal{F}$ can be expressed as \begin{align} d(y^2 + x^n)+x^pU(x)dy=0,\label{takens} \end{align}
where $n\ge 3$,  $p\ge 2$ and $U (0)\neq 0$. An easy computation shows that $\mu(\mathcal{F},0)=n-1$. We organize our exposition in
three cases.

\subsection* {Case 1: $n<2p$.} 

\subsubsection*{Case 1a.} We assume that $n=2 k+1$ for some $k\in\mathbb{N}$. If we blow up the origin,  the exceptional divisor $D_1$ is invariant and contains a unique singularity, which has algebraic multiplicity two and lies in the strict transform of the curve $\{y=0\}$. We blow up this singularity and obtain a second invariant component $D_2$ of the exceptional divisor such that, besides  a reduced singularity at the intersection  $D_2\cap D_1$,  it contains a unique other singularity which has multiplicity two and  still lies in the strict transform of the curve $\{y=0\}$. If we blow up this singularity of multiplicity two, the situation is quite the same: we obtain a third invariant component $D_3$ of the exceptional divisor such that, besides a reduced singularity at $D_3\cap D_2$, it contains a unique other singularity which has multiplicity two  and lies in the strict transform of the curve $\{y=0\}$. So we blow up this singularity, and so on. 
After the $(k+1)^{\textrm{th}}$ blow-up, the situation changes: the last component $D_{k+1}$ is invariant but has no singularities other than the intersection $D_{k+1}\cap D_k$, which has multiplicity two. A final blow-up at this point produces an invariant component $D_{k+2}$
such that, besides the intersections $p_{k}\colon=D_{k+2}\cap D_{k}$ and $p_{k+1}\colon =D_{k+2}\cap D_{k+1}$, it contains
a unique other singularity $p'$, and the foliation is already reduced.  Thus, if we set $p_{j}=D_j\cap D_{j+1}$ for $j=1,\dots, k-1$, the strict transform foliation $\tilde{\mathcal{F}}$ has singularities at the points $p_1,\dots,p_{k+1}$ and $p'$. Moreover we have 
\begin{align}&\label{inicio}\cs (  \tilde{\mathcal{F}}, D_{j+1}, p_j)=-\frac{j}{j+1}\quad \textrm{ for }  j=1\dots, k-1, \\
&\cs ( \tilde{\mathcal{F}}, D_{k+2}, p_k)=-\frac{2k+1}{k},\\
&\cs (  \tilde{\mathcal{F}}, D_{k+2}, p_{k+1})=-\frac{1}{2},\\ 
&\cs ( \tilde{\mathcal{F}}, D_{k+2}, p')=-\frac{1}{4k+2}\end{align}
and the foliation $\mathcal{F}$ is a generalized curve (see \cite{CLS}). 
Let $\tilde{S}$ be the separatrix of the singularity at $p'$ that is transverse to $D_{k+2}$. This curve is the strict transform of the unique separatrix  $S$ of the foliation $\mathcal{F}$. Since $$\cs(\tilde{\mathcal{F}},\tilde{S},p')=-(4k+2),$$ by successively  applications of \eqref{CS-blow-up} we obtain that $$\cs({\mathcal{F}},{S},0)=0.$$ Thus, since the foliation $\mathcal{F}$ is a generalized curve, the Baum-Bott index coincides with the Camacho-Sad index (see \cite{B}) and therefore 
$$\bb (\mathcal{F},0)=0.$$
\subsubsection*{Case 1b.} We assume $n=2 k$ for some $k\in\mathbb{N}$. In this case the reduction is exactly the same as in \emph{case 1a} until the $(k-1)^\textrm{th}$ blow-up. So after this  blow-up,  besides the reduced singularity at  $p_{k-2}=D_{k-2}\cap D_{k-1}$, we obtain a unique other singularity, which has multiplicity two and lies in the strict transform
of the curve $\{y=0\}$. A final blow-up at this singularity produces an invariant component $D_k$, which contains two singularities $p'$ and $p''$ other than $p_{k-1}=D_ {k-1}\cap D_{k}$, and the foliation is already reduced.  If as above we set $p_{j}=D_j\cap D_{j+1}$ for $j=1,\dots, k-2$, then we have 
\begin{align}\label{cstot}&\cs ( \tilde{\mathcal{F}}, D_{j+1}, p_j)=-\frac{j}{j+1}\quad \textrm{ for } j=1\dots, k-1,\\  &\cs ( \tilde{\mathcal{F}}, D_{k}, p')=\cs ( \tilde{\mathcal{F}}, D_{k}, p'')= -\frac{1}{2k}\nonumber \end{align}
and the foliation is again a generalized curve. As we did in case \emph{case 1a}, we can obtain the $BB$ index of $\mathcal{F}$  by doing the computations with the CS indices, but this time we prefer a more general approach, which will work in next cases too: we apply $k$ times the formula  \eqref{BB-blow-up}. From the equations in \eqref{cstot} 
we easily obtain
\begin{align}\label{bbpj} &\bb(\tilde{\mathcal{F}},p_j)=\frac{1}{j+1}-\frac{1}{j}\quad\textrm{ for } j=1,\ldots,k-1,\\
\nonumber &\bb( \tilde{\mathcal{F}}, p')=\bb(  \tilde{\mathcal{F}},  p'')= -2k-\frac{1}{2k}+2.\end{align} 
Besides the singularity $(\mathcal{F},0)$, which has algebraic multiplicity one, each one of the next $k-1$ singularities from
which the components $D_2,\ldots,D_k$ arise have algebraic multiplicity two. So the sum of squares of this $k$ algebraic multiplicities is equal to $4k-3$. Then, after $k$ successively applications of  \eqref{BB-blow-up} we obtain  
\begin{align} \nonumber BB(\mathcal{{F}},0)&=\sum_{j=1}^{k-1}\bb(\tilde{\mathcal{F}},p_j)+ \bb (\tilde{\mathcal{F}},p')+\bb (\tilde{\mathcal{F}},p'')+ 4k-3 \\  \nonumber &=\sum_{j=1}^{k-1}\left(\frac{1}{j+1}-\frac{1}{j}\right)+ \left(-2k-\frac{1}{2k}+2\right)+\\
&\hspace{5cm}+\left(-2k-\frac{1}{2k}+2\right)+4k-3\nonumber \\ 
&=0.\label{bbcase2a}\end{align}

\subsection* {Case 2: $n=2p$.} 
If we set $k=p$, the situation 
is exactly the same as in \emph{case 1} until the $(k-1)^\textrm{th}$ blow-up. After this  blow-up,  besides the reduced singularity at  $p_{k-2}=D_{k-2}\cap D_{k-1}$, we obtain a unique other singularity, which has multiplicity two and lies in the strict transform
of the curve $\{y=0\}$. We blow up this singularity and obtain an invariant component $D_k$ of the exceptional divisor. As in previous case, if we set $p_{j}=D_j\cap D_{j+1}$ for $j=1,\dots, k-1$, 
Equation \eqref{bbpj} also holds true, that is 
\begin{align}\bb(\tilde{\mathcal{F}},p_j)=\frac{1}{j+1}-\frac{1}{j} \quad \textrm{ for  } j=1,\ldots,k-1.\end{align}  In relation with the other singularities in $D_k$, we  find the following two cases.\\

\subsubsection*{Case 2a} Besides $p_{k-1}=D_ {k-1}\cap D_{k}$, there are two other singularities   $p'$ and $p''$   in $D_k$, which are non-degenerate. Set  
\begin{align*} \lambda'=\cs (\tilde{\mathcal{F}},D_k,p')\quad \textrm{and} \quad 
\lambda''=\cs (\tilde{\mathcal{F}},D_k,p'').\end{align*} By Camacho-Sad formula,
$$\lambda'+\lambda''=-\frac{1}{k},$$ so that at least one of this numbers, say $\lambda'$, 
is negative and $p'$ is therefore reduced.  If $\lambda''\notin\mathbb{Q}^+$, the singularity 
at $p''$ is also reduced. If  $\lambda''\in\mathbb{Q}^+$, we need some additional blow-ups to achieve the reduction. After doing this we obtain a dicritical component or a saddle-node, so  the foliation $\mathcal{F}$ is not necessarily a generalized curve. Anyway, since  
\begin{align*} \bb (\tilde{\mathcal{F}},p')=\lambda'+\frac{1}{\lambda'}+2\quad \textrm{and} \quad 
\bb (\tilde{\mathcal{F}},p'')=\lambda''+\frac{1}{\lambda''}+2,
\end{align*} as in previous case we can apply $k$ times the formula  \eqref{BB-blow-up} to compute the Baum-Bott index of $\mathcal{F}$: 
\begin{align} \nonumber BB(\mathcal{{F}},0)&=\sum_{j=1}^{k-1}\bb(\tilde{\mathcal{F}},p_j)+ \bb (\tilde{\mathcal{F}},p')+\bb (\tilde{\mathcal{F}},p'')+ 4k-3 \\  \nonumber &=\sum_{j=1}^{k-1}\left(\frac{1}{j+1}-\frac{1}{j}\right)+ \left(\lambda'+\frac{1}{\lambda'}+2\right)+
\left(\lambda''+\frac{1}{\lambda''}+2\right)+4k-3\\ \nonumber
&=\left(\frac{1}{k}-1\right)+\lambda'+\lambda''+\frac{\lambda'+\lambda''}{\lambda'\lambda''}+4k+1\\
&=\left(\frac{1}{k}\right)-\frac{1}{k}-\frac{1}{k\lambda'\lambda''}+4k\nonumber\\
&=2n-\frac{2}{n\lambda'\lambda''}.\label{bbcase2a}\end{align}
Since $p'$ is reduced, there exists a unique separatrix $\tilde{S'}$ through $p'$ that is transverse to the exceptional divisor. This separatrix is the strict transform of a smooth separatrix $S'$ of $\mathcal{F}$.  We can use Equation \eqref{GSV-blow-up} to compute the $\gsv$ index of $S'$. If we successively blow down the components $D_k, D_{k-1},\dots, D_1$, we successively  obtain $k$ singularities such that the first $k-1$ of them have multiplicity two and the last one --- which is $(\mathcal{F},0)$ --- has multiplicity one. Thus, by applying $k$ times Equation \eqref{GSV-blow-up} we obtain  
\begin{align}\label{gsvs'}\gsv(\mathcal{F},S',0)&=\gsv(\tilde{\mathcal{F}},\tilde{S'},p')+(k-1) \\&=\nonumber1+(k-1)\\
&=\frac{n}{2}.\label{n/2}
\end{align}In the same way,  if $p''$ is reduced, there is a smooth separatrix $S''$ of $\mathcal{F}$ whose strict transform 
pass through $p''$ and proceeding as above we also obtain
\begin{align}\label{n/2'}\gsv(\mathcal{F},S'',0)=\frac{n}{2}.
\end{align}

\subsubsection*{Case 2b} Besides $p_{k-1}=D_ {k-1}\cap D_{k}$, there is exactly one other singularity $p'$ in $D_k$. This singularity is a saddle-node of Milnor number  two, whose  weak separatrix is contained in $D_k$. By Camacho-Sad formula we have $$\cs(\tilde{\mathcal{F}},D_k,p')=-\frac{1}{k},$$ so it follows from
\eqref{saddle-node-bb} that 
$$\bb(\tilde{\mathcal{F}},p')=4-\frac{1}{k}.$$ Proceeding as in previous case,
\begin{align} \nonumber \bb({\mathcal{F}},0)&=\sum_{j=1}^{k-1}\bb(\tilde{\mathcal{F}},p_j)+ \bb (\tilde{\mathcal{F}},p') + 4k-3 \\  \nonumber &=\left(\frac{1}{k}-1\right)+
 \left(4-\frac{1}{k}\right)+4k-3=4k\\ 
&=2n.\label{bbcase2b}\end{align}

\subsection* {Case 3: $n>2p$.} 
Again, if we set $k=p$, the reduction process 
is exactly the same as in previous case until the $(k-1)^\textrm{th}$ blow-up. After this  blow-up,  besides the reduced singularity at  $p_{k-2}=D_{k-2}\cap D_{k-1}$ we obtain a unique other singularity, which has multiplicity two and lies in the strict transform
of the curve $\{y=0\}$. We blow up this singularity and the reduction is finished: the last component $D_k$ of the
exceptional divisor is invariant and contains two singularities   $p'$ and $p''$   other than $p_{k-1}= D_{k-1}\cap D_k$.  The singularity at
$p'$ is resonant non-degenerate  and we have 
\begin{align}\label{cs=0}\cs(\tilde{\mathcal{F}},D_k,p')=-\frac{1}{k}.\end{align} The singularity at $p''$ is a saddle-node whose strong separatrix is contained in $D_k$; in particular $$\cs(\tilde{\mathcal{F}},D_k,p'')=0.$$ 
In relation with the other singularities, 
Equation \eqref{inicio} also holds true. Let $\tilde{S'}$ be  the separatrix through $p'$  that is transverse to $D_k$ and let $S'$ denote 
 the separatrix of $\mathcal{F}$ defined 
by $\tilde{S'}$; the separatrix $S'$ is smooth. 
From Equation \eqref{cs=0} we see that
$\cs(\tilde{\mathcal{F}},\tilde{S'},p')=-k$, so by applying $k$ times Equation \eqref{CS-blow-up} we obtain
\begin{align}\label{cscase3}\cs(\mathcal{F},S',0)&=\cs(\tilde{\mathcal{F}},\tilde{S'},p')+k\\
&=0.\nonumber\end{align} In relation with the $\gsv$ index, 
 proceeding as in \emph{Case 2b} we arrive to the same equation as that in \eqref{gsvs'}, that is
\begin{align}\gsv(\mathcal{F},S',0)&=\gsv(\tilde{\mathcal{F}},\tilde{S'},p')+(k-1)\nonumber\\
&=k.\label{gsvk0} \end{align}
Suppose  the weak separatrix $\tilde{S''}$ of the saddle-node at $p''$ is convergent. Then, proceeding in the same way as above we obtain 
\begin{align}\gsv(\mathcal{F},S'',0)=\gsv(\tilde{\mathcal{F}},\tilde{S''},p')+(k-1),
\end{align} but this time we have $\gsv(\tilde{\mathcal{F}},\tilde{S''},p')=\mu$, where $\mu\ge 2$ is the Milnor number of the saddle-node at $p''$. Therefore  
 \begin{align}\label{gsvk}\gsv(\mathcal{F},S'',0)=\mu+k-1.
\end{align} Since  $\mu(\mathcal{F},0)=n-1$, we can use $k$ times the formula  \eqref{Milnor-blow-up}   to relate $k$,  $\mu$ and  $n$, whence we obtain --- we left the details to the reader ---
\begin{align} \label{n,k} n=\mu+2k-1.\end{align}

\section{Proof of Theorem \ref{teorema1}}\label{pro}
Let $\mathcal{F}$ be a holomorphic foliation of degree $d$  on $\mathbb{P}^2$ with only one singularity. By \eqref{Milnorsum} the singularity of $\mathcal{F}$ has Milnor number equal to $d^2+d+1$. We recall that a singularity whose algebraic multiplicity is equal to one necessarily belongs to some of the following types:
\begin{enumerate}
\item The singularity is non-degenerate; so its Milnor number is equal to one.
\item The singularity is a saddle-node.
\item The singularity is nilpotent.
\end{enumerate} Since $d^2+d+1>1$, only the two last cases are possible. Thus, we reduce the proof of Theorem  \ref{teorema1} 
to proving the propositions \ref{saddle-node-case} and \ref{nilpotent-case} below. It is known  that saddle-nodes or  nilpotent singularities can only appear if $d\ge 2$. Thus, since Theorem \ref{teorema1} is 
already proved for $d=2$ in \cite{CD},  we perform the proofs of  propositions \ref{saddle-node-case} and \ref{nilpotent-case} under the assumption that  $d\ge 3$. 

\begin{pro}\label{saddle-node-case}Let $\mathcal{F}$ be a holomorphic foliation of degree $d$  on $\mathbb{P}^2$ with only one singularity. If this singularity is a saddle-node, then $\mathcal{F}$ has no invariant algebraic curve. 
\end{pro}

\begin{proof} Suppose that there exists an algebraic curve $C$ of degree $m$ that is invariant by $\mathcal{F}$. Let $p\in\mathbb{P}^2$ be the singularity of $\mathcal{F}$. Let $S_w$ and $S_s$ be the weak and
strong separatrices of $p$, respectively.  From \eqref{CSformula} we have 
$$\operatorname{CS}(\mathcal{F},C,p)=m^2>0.$$ Thus, since that $\operatorname{CS}(\mathcal{F},S_s,p)=0$, the germ of $C$ at $p$ only admits the following two possibilities:
\begin{enumerate}
\item $(C,p)=S_w$; in this case  $$\gsv(\mathcal{F},C,p)=\mu(\mathcal{F},p)=d^2+d+1.$$
\item  $(C,p)=S_w\cup S_s$; in this case \begin{align*}\gsv(\mathcal{F},C,p)&=\gsv(\mathcal{F},S_w,p)+
\gsv(\mathcal{F},S_s,p)-2(S_w,S_s)_p\\ &=\mu(\mathcal{F},p) +1 -2 (1)
=d^2+d.\end{align*}
\end{enumerate}So in any case we have   \begin{align}\label{gsvd2}\gsv(\mathcal{F},C,p)\ge d^2+d.\end{align}  
On the other hand, by \eqref{GSVformula} we have 
\begin{align*}  \gsv(\mathcal{F},C,p)&=m(d+2)-m^2=m(d+2-m)\\
 &\le\left(\frac{d+2}{2}\right)^2.\end{align*} This inequality together with \eqref{gsvd2} give us 
$$d^2+d\le \left(\frac{d+2}{2}\right)^2,$$ which is impossible for  $d\ge 3$. 

 \end{proof}

\begin{pro}\label{nilpotent-case}Let $\mathcal{F}$ be a holomorphic foliation of degree $d$  on $\mathbb{P}^2$ with only one singularity. If this singularity is nilpotent, then $\mathcal{F}$ has no invariant algebraic curve. 
\end{pro}
\begin{proof} Let $p\in\mathbb{P}^2$ be the singularity
of $\mathcal{F}$. As we have seen in Section \ref{nilpotent}, in suitable holomorphic coordinates the singularity at $p$ has the expression \eqref{takens}.  Then we have $\mu(\mathcal{F},p)=n-1$ and therefore   $$n=d^2+d+2.$$  By \eqref{BaumBottsum} we have 
$$\bb(\mathcal{F},p)=(d+2)^2>0. $$ Then it is obvious that  $p$ can not correspond to 
{Case 1} of Section \ref{nilpotent}. Neither can $p$ correspond to Case 2b  of Section \ref{nilpotent}: otherwise 
\begin{align*}(d+2)^2&=2n\\
&=2(d^2+d+2),\end{align*} which is impossible for $d\ge 3$.  Then we have the following two possibilities:
\begin{enumerate}[(i)]
\item\label{i} The singularity at $p$ corresponds to case 2a of Section \ref{nilpotent}. Furthermore,
  both non-degenerate singularities $p'$ and $p''$ 
appearing after the $k^{\textrm{th}}$ blow-up   are already reduced: otherwise, we have $\lambda'\lambda''<0$ and using this together with \eqref{bbcase2a} we obtain
\begin{align*}(d+2)^2&=\bb(\mathcal{F},q)=2n-\frac{2}{n\lambda'\lambda''}\\
&>2(d^2+d+2),
\end{align*}  which is impossible for $d\ge 3$. 
\item \label{ii}The singularity at $p$ corresponds to case 3 of Section \ref{nilpotent}.  The reduction is the same as in  previous case, 
but only $p'$ is non-degenerate while  $p''$ is a  saddle-node.

\end{enumerate}
Let $\tilde{S'}$ and $\tilde{S''}$ be respectively the separatrices of $p'$ and $p''$ that are transverse to the exceptional divisor; the curve $\tilde{S''}$ could be purely formal if we are in case \eqref{ii}. Let $S'$ and $S''$ be the separatrices of $(\mathcal{F},p)$ defined by $\tilde{S'}$ and $\tilde{S''}$, respectively. Let $C\subset \mathbb{P}^2$ be an algebraic curve of degree $m$ which is invariant by $\mathcal{F}$.  \\

Suppose first that $C$ is smooth. Then  the germ $(C,p)$ 
 of $C$ at $p$ is equal to
 $S'$ or $S''$. If we are in case \eqref{i}, we have from \eqref{n/2} and \eqref{n/2'} that 
 \begin{align}\label{n21}\cs(\mathcal{F},C,p)=\frac{n}{2}.
 \end{align}
 If we are in case \eqref{ii}, by \eqref{cscase3} we have  $\cs(\mathcal{F},S',p)=0$. Then, since Camacho-Sad formula
 \eqref{CSformula} gives $\cs(\mathcal{F},C,p)=m^2>0$, we conclude  that $(C,p)=S''$. Thus, it follows from equations \eqref{gsvk} and \eqref{n,k} 
 that \begin{align}\label{n22}\gsv(\mathcal{F},C,p)\ge \frac{n}{2}.\end{align} From \eqref{n21} and \eqref{n22} in any case we have  \begin{align}\label{n23}\gsv(\mathcal{F},C,p)\ge \frac{d^2+d+2}{2}.\end{align}
 Since $p$ is the only singularity of $\mathcal{F}$, by Equation \eqref{GSVformula} we have
\begin{align*}\gsv (\mathcal{F}, C,p)&=m(d+2)-m^2\\
&=m(d+2-m)\le \left(\frac{d+2}{2}\right)^2.
\end{align*} Combining this inequality with that in \eqref{n23} we obtain $$\frac{d^2+d+2}{2}\le \left(\frac{d+2}{2}\right)^2,$$ which is impossible for $d\ge 3$.  \\

Suppose now that $C$ is not smooth. Then we have $(C,p)=S'\cup S''$. Since the branches $S'$ and $S''$ are smooth, we have 
$$\gsv (\mathcal{F}, S', p)=\mu(\mathcal{F},S',p),\quad \gsv (\mathcal{F}, S'',p)=\mu(\mathcal{F},S'',p)$$
and by Equation \eqref{CLformula} we obtain
\begin{align}\label{cln}\gsv (\mathcal{F}, S', p)+\gsv (\mathcal{F}, S'',p)=\chi (\tilde{C})+m(d-1).\end{align} It follows from Equation \eqref{n21} in case \eqref{i} and
Equations \eqref{gsvk0}, \eqref{gsvk} and \eqref{n,k} in case \eqref{ii} that the left side of \eqref{cln} is equal to $n$. 
Moreover, since $\mathcal{F}$ is non-dicritical,  by \cite{carnicer} we have  $m\le d+2$. Therefore 
 \begin{align*}\chi (\tilde{C})&= n-m(d-1)\ge (d^2+d+2)-(d+2)(d-1)=4,\end{align*} which is impossible for a compact Riemann surface. 
 \end{proof}

  \section{Examples}\label{exa}
  In this section we present several examples of foliations of degree 3 of   $\mathbb{P}^2$
  with a unique singularity. 
  \subsection{Saddle-node examples}

Let $\mathcal{F}$ be a degree 3 foliation of  $\mathbb{P}^2$ with a unique singularity, which is supposed to be a saddle-node.  On suitable affine coordinates $\mathcal{F}$ is  given by the 1-form
\begin{equation*}\begin{aligned}\label{eqsillas}
(x+A_2(x,y)+A_3(x,y)-x \phi(x,y))dy+(B_2+B_3+y\phi(x,y))dx, 
\end{aligned}
 \end{equation*}
where
$A_2=a_{20}x^2+a_{11}xy+a_{02}y^2$, $B_2=b_{20}x^2+b_{11}xy+b_{02}y^2$, $A_3=a_{30}x^3+a_{21}x^2y+a_{12}xy^2+a_{03}y^3$ , $B_3=b_{30}x^3+b_{21}x^2y+b_{12}xy^2+b_{03}y^3$ and
 $\phi=p_{30}x^3+p_{21}x^2y+p_{12}xy^2+p_{03}y^3$.
Now we set  $B_3 \equiv 0$. Then it is not difficult to find  a projective change of coordinates such that 
  the foliation $\mathcal{F}$ is defined by 
  the 1-form 
  \begin{equation}\label{eqoms}
\omega= \big[x+A_2(x,y)+xA_3^2(x,y)-x \phi(x,y)\big]dy+\left[x^2+y\phi(x,y)\right]dx, 
\end{equation}
where $A_j$ and $\phi$ are  homogenous polynomials of degree $j$ and $3$ respectively.
Then we have the following result.
 \begin{pro}\label{propro}
Let $\mathcal{F}$ be a singular holomorphic foliation on $\mathbb{P}^2$ of degree $3$, given by a 1-form as in 
\eqref{eqoms}.
Then the saddle-node at the origin is the only singularity of  $\mathcal{F}$  if and only if  $\omega$ is one of the following 1-forms:
\begin{enumerate}
 \item $\omega_a=(x+ay^2-ax^2y-x\phi(x,y))dy+(x^2+y\phi(x,y)) dx,$ where $\phi=-4ax^3-a^2y^3$ for some $a \in \mathbb{C}^*$. The foliations defined by these 1-forms are pairwise equivalent by transformations of the form $$(x,y)\mapsto(\lambda x, \lambda^2 y),\quad\lambda\in\mathbb{C}^*.$$  
\item $\omega_{a,c}=(x+A_2(x,y)+xA_3^2(x,y)-x \phi(x,y))dy+(x^2+y\phi(x,y))dx$ with
\begin{align*}
A_2(x,y)&=\dfrac{5c}{a} x^2 + \dfrac{3c^2}{a^2} x y + a y^2,\\
A_3^2(x,y)&=\dfrac{15c^2}{a^2}x^3+14ax^2y+cxy^2,\\
\phi(x,y)&=-64 a x^3 - 35 c x^2 y - \dfrac{18 a^3}{c} x y^2 - a^2 y^3,
\end{align*}
where $a\in\mathbb{C}^*$ and $c^3=3a^4$. The foliations defined by these 1-forms are pairwise equivalent by transformations of the form
 $$(x,y)\mapsto(\alpha^3\beta^2 x, \alpha^2\beta y),\quad \alpha,\beta
 \in\mathbb{C}^*\alpha^4\beta^3=1.$$
\end{enumerate}
Moreover, the foliations defined by $\omega_a$ and $\omega_{a,c}$ are not equivalent. 

\end{pro}
\begin{proof} We only outline the proof and omit several long calculations easily done with a computer. 
By the Implicit Function Theorem the curve $A=0$ near the origin is parametrized by $(x(t),t)$, where
\[
x(t)=c_2t^2+c_3t^3+c_4 t^4+\cdots c_{11}t^{11}+\cdots.
\]
Since $A(x(t),t)=0$ we obtain

\begin{align*}
c_2 & = -a_{20},\qquad  c_3=-a_{11}c_2,\qquad c_4=-a_{21}c_2-a_{20}c_2^2-a_{11}c_3, \\
c_5 &=-a_{12}c_2^2-a_{21}c_3-2a_{20}c_2c_3-a_{11}c_4+c_2p_{03},\\
c_6 &=-a_{03}c_2^3-2a_{12}c_2c_3-a_{20}c_3^2-a_{21}c_4-2a_{20}c_2c_4-a_{11}c_5+c_3 p_{03}+c_2^2p_{12},\\
&\hspace{2mm}\vdots \\
&\hspace{-5.5mm}\begin{aligned}
c_{11} = -a_{11} c_{10} - 3 a_{03} c_3 c_4^2 - 3 a_{03} c_3^2 c5 - 6 a_{03} c_2 c_4 c_5 - 
   a_{12} c_5^2 -a_{21}c_5^2- 6 a_{03} c_2 c_3 c_6 -\\
    \hspace{10mm} 2 a_{12} c_4 c_6 -2a_{20}c_5c_6- 3 a_{03} c_2^2 c_7 - 
   2 a_{12} c_3 c_7 - 
   -2a_{20}c_4c_7-2 a_{12} c_2 c_8 -\\ 
   2a_{20}c_3c_8- a_{21} c_9 -2a_{20}c_2c_9+c_8p_{03} + 
   2 c_4 c_5 p_{12} + 2 c_2 c_7 p_{12} + 3 c_3^2  c_4 p_{21} +\\ 
    6c_2c_3c_5p_{21}+3c_2^2 c_6p_{21}+4 c_2 c_3^2 p_{30} + 12 c_2^2 c_3c_4 p_{03} + 4c_2^3c_5 p_{30},
   \end{aligned}
\end{align*}
that is, 
$$c_j=c_j(a_{20},a_{11},a_{02},a_{30},a_{21},a_{12},p_{30},p_{21},p_{12},p_{03}),\quad j=2, \ldots, 11.$$
Then 
\[B(x(t),t)= d_2 t^2+d_3t^3+\cdots+d_{13}t^{13}+\cdots,\]
where $d_j=d_j(c_2,c_3,c_4,\cdots,c_{12})$ for $j=2, \dots,13$. We must find conditions on the coefficients of $A$ and $ B$ such that
$d_j=0$ for all $j=2, \cdots,12$ and $d_{13} \neq 0$. Note that $a_{02} \neq 0$, otherwise in~\eqref{eqoms} we have that $L=\{x=0\}$ is an invariant line.
Straightforward calculations allows us to obtain the following conclusions:
\begin{align*}
d_2&=d_3=0,\\
d_4&= 0 \iff p_{03}=-a_{02}^2,\\
d_5&= 0 \iff p_{12}=-2a_{02} a_{11},\\
d_6&= 0 \iff p_{21}=-a_{11}^2+2a_{12}-2a_{02}a_{20},\\
d_7&= 0 \iff p_{30}=-2(a_{02}+a_{11}a_{20}-a_{21}),\\
d_8&= 0 \iff a_{30}=\frac{2a_{02}^2a_{11}-a_{12}^2+a_{02}^2a_{20}^2}{2a_{02}^2},\\
d_9&= 0 \iff a_{20}=\frac{-a_{11}a_{12}^2+a_{02}a_{12}a_{21}}{a_{02}^3},
\end{align*}
and $d_{10}=0$ is equivalent to
\begin{equation}\label{eq:d10}
a_{02}(a_{02}^4+a_{12}^3)a_{21}=-a_{02}^6+3a_{02}^4a_{11}a_{12}+a_{02}^2a_{12}^3+a_{11}a_{12}^4.
\end{equation}
We claim that $a_{02}^4+a_{12}^3\neq 0$. Assume that $a_{02}^4+a_{12}^3=0$.  Then $a_{12} \neq 0$ and~\eqref{eq:d10} simplifies to
\begin{equation}\label{eq:d10s}
0=a_{02}(a_{02}^2-a_{11}a_{12}) a_{21}+a_{11}^2a_{12}^2+a_{12}^3. 
\end{equation}
If $a_{02}^2-a_{11}a_{12}=0$ then $a_{11}=\dfrac{a_{02}^2}{a_{12}}$, which allow us
to obtain a simplified  expression of $d_{12}$ as $$d_{12}=a_{02} a_{12}^2 (-2 a_{02} + a_{21})^3.$$ Then $d_{12}=0$ implies $a_{21}=2a_{02}$ and this  equality simplifies the expression of $d_{13}$ to finally obtain $d_{13}=0$. If, on the other hand, $a_{02}^2-a_{11}a_{12}\neq 0$, then equation~\eqref{eq:d10s} implies 
$$a_{21}=\dfrac{-a_{11}^2a_{12}^2-a_{12}^3}{a_{02}(a_{02}^2-a_{11}a_{12})},$$ 
which also leads to $d_{13}=0$. In both cases the Milnor number is not 13. Therefore $a_{02}^4+a_{12}^3 \neq 0$ and 
\[
 d_{10} = 0 \iff a_{21}=\frac{-a_{02}^6+3a_{02}^4a_{11}a_{12}+a_{02}^2a_{12}^3+a_{11}a_{12}^4}{a_{02}(a_{02}^4+a_{12}^3)}.
\]
By working with the expressions of $d_{11},d_{12}$ and $d_{13}$ we can see that  $a_{02}^4-a_{11}^3 \neq 0$,  otherwise we would obtain $d_{11}=d_{12}=d_{13}=0$.  Then we can prove that
\begin{align*}
 &d_{11}=0 \iff a_{11}=\frac{-3a_{02}^4a_{12}^2-a_{12}^5}{a_{02}^4-a_{12}^3},\textrm{ and }\\
& d_{12}=0 \iff a_{12}=0 \quad \textrm{ or } \quad 3a_{02}^4-a_{12}^3=0.
\end{align*}
Let us analize each case separately:
\begin{itemize}
 \item $a_{12}=0$ : In this case $d_{11}=d_{12}=0$ and $d_{13}=-4a_{02}^8\neq 0$, so 
 we obtain
 \[
  \omega=\left[x+a_{02}y^2-a_{02}x^2y-x(-4a_{02}x^3-a_{02}^2y^3)\right]dy+\left[x^2+y(-4a_{02}x^3-a_{02}^2y^3\right] dx
 \]
\item $3a_{02}^4-a_{12}^3=0$: In this case $d_{11}=d_{12}=0$ and $d_{13}=\frac{-64 a_{12}^3}{9}\neq 0$, so we obtain
\begin{multline*}
 w=\Bigg[x+
 \frac{5 a_{12}x^2}{a_{02}}+\frac{3a_{12}^2xy}{a_{02}^2}+a_{02} y^2\\ +x\bigg(\frac{15a_{12}^2x^2}{a_{02}^2}+14a_{02}xy+a_{12}y^2\bigg)\\
 +x\bigg( 64a_{02} x^3+a_{02}^2y^3+ 35 a_{12} x^2y+\frac{18 a_{02}^3 xy^2}{a_{12}}\bigg) \Bigg]dy\\
 + \Bigg[x^2-y\bigg( 64a_{02} x^3+a_{02}^2y^3+ 35 a_{12} x^2y+\frac{18 a_{02}^3 xy^2}{a_{12}}\bigg)\Bigg]dx.
\end{multline*}
\end{itemize}
The remaining assertions of the proposition are not difficult to prove, so we leave the verification to the reader.

\end{proof}

\begin{exa}\label{ej4} If $b\in\mathbb{C}\backslash\{1,2\}$,  $\alpha=\frac{1+b-b^2}{b-1}$, $\beta=\frac{(b-2)^2}{1-b}$ and $\gamma=1-b$, the 1-form
\begin{align*}
   \omega_b=(x + y^2+ \alpha x^2 y+x(\beta x^3+\gamma y^3))dy + 
  (x^2 +bx y^2-y(\beta x^3+\gamma y^3))dx
  \end{align*}
   defines a foliation on $\mathbb{P}^2$ with a unique singularity which is saddle-node. Note that
   $b=0$ gives the 1-form $\omega_1$ of Proposition \ref{propro}.
 As we see next, the family of 1-forms $\{\omega_b\}$ define pairwise non-equivalent foliations of $\mathbb{P}^2$. Note first that the saddle-node at the origin of 
 $\omega_b$ has $\{y=0\}$ and $\{x=0\}$ as strong and weak separatrices respectively. It is easy to see that the foliation $\omega_b=0$ has order 2 tangencies with $\{y=0\}$ and
 $ \{x=0\}$ at the origin. Thus, since the foliation has degree 3, we conclude that the foliation 
  is tangent to $\{y=0\}$ and $\{x=0\}$ at the infinity and there are no other
 tangencies with these lines outside the saddle-node. Therefore, if $T$ is an automorphism conjugating $\omega_{b_1}=0$ with $\omega_{b_2}=0$, then $T$ has to preserve the
 lines $\{y=0\}$ and $\{x=0\}$  and their points at infinity. This means that $T$ is induced by a map of the form $h\colon (x,y)\mapsto (ax,by)$. Now it is easy to compute 
 $h^*(\omega_{b_2})$ and conclude that $h^*(\omega_{b_2})\in \mathbb{C}^*\omega_{b_1}$ only if $ b_1=b_2$. 
  \end{exa}

\begin{exa}\label{ej5.sn}
Consider the 1-form
\begin{align*}
   \omega=\left[x + A_2(x,y)+ xA_3^2(x,y)+x\phi(x,y)\right]dy + 
  \left[tx^2+B_3(x,y) -y\phi(x,y)\right]dx, 
  \end{align*}
  with
  \begin{align*}
  &A_2(x,y)=\alpha x^2+wxy+ry^2\\
  &A_3^2(x,y)=\dfrac{t w(4 r t-3s)}{s} x^2 + \beta x y + r\alpha y^2\\
 & B_3(x,y)=\dfrac{t w(3 s - 2 r t)}{s}x^2 y + s x y^2\\
 & \phi(x,y)=\dfrac{t (2rt-s)^2}{(rt-s)}x^3 
  + \dfrac{2 t (r t-s) w^2}{s}x^2 y + \dfrac{(r t-s) (s + 2 r t) w }{s}x y^2 + 
 r (r t-s) y^3,\\
  &\alpha =  \dfrac{t w^2(rt-s)}{s^2},\\
  &   \beta = 
  \frac{1}{s^2 (r t-s)}(s^4 - r s^3 t - r^2 s^2 t^2 + s^2 t w^3 - 2 r s t^2 w^3 + r^2 t^3 w^3),
  \end{align*}  
   where
  $r,s,t\in\mathbb{C}^*$, $w\in \mathbb{C}$, $rt\neq s$. The 1-form $\omega$
  define a foliation on $\mathbb{P}^2$ with a unique singularity which is a saddle-node.
If we set $w=0$ the 1-form $\omega$ adopt the form
\begin{align*}
   \left[x + ry^2+\beta x^2 y+x\phi(x,y)\right]dy + 
  \left[tx^2+s x y^2 -y\phi(x,y)\right]dx, 
  \end{align*}
  where 
\[  
\phi(x,y)=\dfrac{t (2rt-s)^2}{(rt-s)}x^3 + r (r t-s) y^3,
\]
 
  \[
  \beta = 
  \frac{1}{r t-s}(s^2 - r s t - r^2 t^2)
  \]
and $r,s,t\in\mathbb{C}^*$, $rt\neq s$. This family is equivalent to that given by Alc\'antara and Pantale\'on-Mondrag\'on in \cite{ClP}. 
Also note that if $w=0$, $r=t=1$ and $s=b$  we recover Example~\ref{ej4}.
\end{exa}

 \subsection{Nilpotent examples}
Finally, we give some examples of degree 3 foliations of $\mathbb{P}^2$ with a unique singularity which is  nilpotent. As in previous case, we have done several calculations with
 the aid of a computer, so we just present the examples omitting the verifications.  
\begin{exa}\label{ej5}
Given  $a=(a_1,a_2,a_3,a_4)\in\mathbb{C}^4$ with $a_4\neq 0$, then
\begin{align*}
\omega_a=\left[y+x^2+2a_1xy+ a_{1}x^3+a_{2}x^2y-a_3xy^2
-x \phi(x,y)\right] dy+\\
\left[xy+x^3+a_{1}x^2y+(a_{1}^2-a_{2})xy^2+a_3y^3+y \phi(x,y)\right]dx,
 \end{align*}
where $$\phi(x,y)= (a_{1}^2-a_{2})x^3+(-a_{2}a_{1}+a_{1}^2+a_3)x^2y+a_{1}a_3xy^2+a_4y^3$$ defines a degree 3 foliation with a nilpotent singularity
at the origin  which is the only singularity on $\mathbb{P}^2$. Moreover, given $a,\tilde{a}\in\mathbb{C}^4$, the foliations defined by $\omega_a$
 and $\omega_{\tilde{a}}$ are equivalent if and only if there exists $\lambda\in\mathbb{C}^*$
 such that $$(\tilde{a}_1,\tilde{a}_2,\tilde{a}_3,\tilde{a}_4)=(\lambda a_1, \lambda^2 a_2,
 \lambda^3a_3,\lambda^5 a_4).$$

\end{exa}

\begin{exa}\label{ej5}
Given $a=(a_{1},a_2,a_3)\in\mathbb{C}^3$, $a_2,a_3\in\mathbb{C}^*$, then
 \begin{equation*}
 \omega_a=\big(y+A_2(x,y)+A_3(x,y)-x\phi(x,y)\big)dy+ \big(xy + B_3(x,y)+y \phi(x,y) \big)dx,
 \end{equation*}
where
\begin{align*}
A_2(x,y)&= x^2 + (2 a_{1} + a_2) x y\\
A_3(x,y)&=a_{1} x^3 + \frac{1}{3} (4 a_{1}^2 + 3 a_{1} a_2 + 3 a_3) x^2 y +\\
&\qquad\quad  \frac{1}{54} a_{1} (2 a_{1}^2 - 9 a_2^2 + 18 a_3) x y^2 - 
 \frac{1}{36} a_2^2 (a_{1}^2 + 18 a_3) y^3\\
B_3(x,y)&=x^3 + (a_{1} + a_2) x^2 y - \frac{1}{3} (a_{1}^2 + 3 a_3) x y^2 - 
 \frac{1}{27} a_{1} (a_{1}^2 + 9 a_3) y^3\\
\phi(x,y)&= -\frac{1}{6} (2 a_{1}^2 + a_{1} a_2 + 6 a_3) x^3 -\frac{1}{324} a_2(2 a_{1}^4 + 36 a_{1}^2 a_3+ 81 a_3^2)  y^3-\\
 &\qquad  
 \frac{1}{54} (20 a_{1}^3 + 27 a_{1}^2 a_2+ 9a_{1} (a_2^2 + 8  a_3) + 81 a_2a_3) x^2 y -\\
 &\qquad \frac{1}{108} (4 a_{1}^4 + 10 a_{1}^3 a_2 + 3a_{1}^2(a_2^2 + 12 a_3) + 
    90 a_{1} a_2a_3 + 54 a_2^2 a_3) x y^2
\end{align*}
     defines a degree 3 foliation with a nilpotent singularity
at the origin  which is the only singularity on $\mathbb{P}^2$. Moreover, given $a,\tilde{a}\in\mathbb{C}^3$, the foliations defined by $\omega_a$
 and $\omega_{\tilde{a}}$ are equivalent if and only if there exists $\lambda\in\mathbb{C}^*$
 such that $$(\tilde{a}_1,\tilde{a}_2,\tilde{a}_3)=(\lambda a_1, \lambda a_2,
 \lambda^2a_3).$$

\end{exa}

\end{document}